\documentclass[12pt,a4paper]{article}

\usepackage[a4paper,left=2.5cm,right=2.5cm,
top=2.5cm,bottom=2.5cm]{geometry}

\usepackage{amsmath,amssymb,amsfonts,amsthm}
\usepackage{mathrsfs}
\usepackage{graphicx}
\usepackage{booktabs}
\usepackage{multirow}
\usepackage{xcolor}
\usepackage{enumerate}
\usepackage{appendix}
\usepackage{hyperref}

\newtheorem{theorem}{Theorem}[section]

\newtheorem{lemma}[theorem]{Lemma}

\theoremstyle{definition}
\newtheorem{definition}[theorem]{Definition}

\theoremstyle{remark}

\numberwithin{equation}{section}

\title{\bfseries
 Generalized Chen's inequalities for Riemannian submersions and
Riemannian maps with Applications
}

\author{
Ravindra Singh
\\[2ex]
Department of Mathematics,\\ Banaras Hindu University,\\
 Varanasi 221005, India\\
E-mail: khanelrs@bhu.ac.in\\
ORCID: 0009-0009-1270-3831
}

\date{}

\begin{document}

\maketitle

\begin{abstract}
\noindent In this paper, we establish generalized B.-Y. Chen inequalities for Riemannian submersions and Riemannian maps between Riemannian manifolds by employing the generalized $\delta$-invariants introduced by Chen. We derive optimal inequalities involving the generalized $\delta$-invariants associated with the vertical and horizontal distributions together with extrinsic invariants determined by the second fundamental tensors of the submersion and the Riemannian map. Furthermore, we characterize the equality cases through precise algebraic conditions on the corresponding shape operators, providing their geometric interpretation. As applications, we obtain explicit generalized Chen inequalities for Riemannian submersions and Riemannian maps whose total or target manifolds are real space forms and complex space forms. These results extend the classical Chen inequalities as well as several recent Chen-type inequalities for Riemannian submersions and Riemannian maps available in the literature.
\end{abstract}

\noindent{\bf Keywords:}
Riemannian manifold; space form; sectional curvature;
scalar curvature; Casorati curvature; Riemannian submersion

\tableofcontents

\section{Introduction}
B.-Y. Chen introduced, several years ago, a fundamental Riemannian invariant
denoted by $\delta $. In \cite{Chen_93_AM}, he established a general sharp
inequality relating this invariant to the squared mean curvature of
arbitrary submanifolds in Riemannian space forms. Upon applying this
inequality, Chen provided solutions to some problems (see \cite%
{Chen_93_AM,Chen_95_RM,Chen_00_JJM}). Submanifolds attaining equality in
this inequality have been the subject of extensive study in recent years
(see \cite{Chen_2011}). In \cite{Chen_00_JJM}, Chen further extended the
invariant $\delta $ on an $r$-dimensional Riemannian manifold to two
sequences of Riemannian invariants, denoted by $\delta (r_{1},\ldots ,r_{k})$
and $\widehat{\delta }(r_{1},\ldots ,r_{k})$, where the integers $%
r_{1},\ldots ,r_{k}$ satisfy 
\[
2\leq r_{1}\leq r_{k}<r. 
\]%
He also generalized the original sharp inequality in \cite{Chen_93_AM} to
new inequalities involving these extended invariants.

For integers $r\geq 2$ and $k\geq 0$, let ${\cal S}(r;k)$ denote the finite
set of all $k$-tuples $(r_{1},\ldots ,r_{k})$ satisfying 
\[
2\leq r_{1}\leq r_{k}<r. 
\]%
In particular, 
\[
{\cal S}(r,0)=\{\emptyset \}, 
\]%
and we define 
\[
{\cal S}(r)=\bigcup_{k\geq 0}{\cal S}(r;k). 
\]%
The two families of Riemannian invariants 
\[
\delta (r_{1},\ldots ,r_{k})(p)\quad \text{and}\quad \widehat{\delta }%
(r_{1},\ldots ,r_{k})(p), 
\]%
for $(r_{1},\ldots ,r_{k})\in {\cal S}(r)$, are defined at a point $p\in M_{%
\text{total}}$ by 
\[
2\delta (r_{1},\ldots ,r_{k})(p)=\tau (p)-\inf \{\tau (L_{1})+\cdots +\tau
(L_{k})\}, 
\]%
\[
2\widehat{\delta }(r_{1},\ldots ,r_{k})(p)=\tau (p)-\sup \{\tau
(L_{1})+\cdots +\tau (L_{k})\}, 
\]%
where $L_{1},\ldots ,L_{k}$ are mutually orthogonal subspaces of $T_{p}M_{%
\text{total}}$ such that 
\[
\dim L_{j}=r_{j},\qquad j=1,\ldots ,k. 
\]%
Clearly, 
\[
\delta (r_{1},\ldots ,r_{k})\geq \widehat{\delta }(r_{1},\ldots ,r_{k}). 
\]%
In this setting, the scalar curvature $\tau $ corresponds to 
\[
\delta (\emptyset )=\widehat{\delta }(\emptyset ) 
\]%
(i.e., $k=0$), while the original invariant $\delta $ introduced in \cite%
{Chen_93_AM,Chen_95_RM} is precisely $\delta (2)$ (with $k=1$ and $r_{1}=2$%
). It is worth noting that the invariants $\delta (r_{1},\ldots ,r_{k})$ and 
$\widehat{\delta }(r_{1},\ldots ,r_{k})$ with $k>0$ differ essentially in
nature from the scalar curvature.\newline
Riemannian submersions are smooth maps and have applications in various
fields, including physics, mechanics, relativity, spacetime, robotics,
supergravity, superstring, Kaluza-Klein, Yang-Mills theories, etc. The
concept of a Riemannian map generalizes that of a submanifold and a
Riemannian submersion \cite{Fischer_1992}. It has rich geometry and
applications \cite{GRK_book, Sahin_book}, and satisfies the eikonal equation
(a bridge between geometric and physical optics). In addition, Riemannian
maps allow us to compare the geometric properties of the source and target
Riemannian manifolds. Due to the interesting properties mentioned above,
Riemannian submersions and Riemannian maps were investigated with various
structures (see \cite{Falcitelli_04_WS, Sahin_book} and references therein).
Motivated by Chen's inequalities for submanifolds, several authors have
established analogous results for Riemannian submersions by relating
intrinsic invariants with extrinsic invariants induced by the submersion and
Riemannian map. In this direction, Recently, Singh and Tripathi \cite%
{Singh_Tripathi_26} introduced the concept of B.-Y. Chen's inequality for
Riemannian submersion between two Riemannian manifold and their applications
to Riemannian submersion whose total space are real, complex and generalized
Sasakian space forms. For Riemannian map, the Chen's first (B.-Y. Chen's)
inequality obtained to real space form by Sahin \cite{Sahin_2016}, for
complex space form \cite{Meena}. After that many researcher obtained this
inequality for Riemannian maps. Singh et al. \cite{SMM_CF_2025} obtained
Chen's first (B.-Y. Chen's) inequality for Riemannian maps between two
Riemannian manifold and its application for Riemannian maps to generalized
complex and generalized Sasakian space forms. Motivated by the above study,
in this paper we establish generalized Chen's inequality for Riemannian
submersion and Riemannian maps between two Riemannian manifolds and their
application to real space form and complex space form. \newline
This paper is organized as follows. In section \ref{sec 2}, we give some
basic definitions and notations related to Riemannian submersions. In
section \ref{sec 3}, we establish we establish generalized Chen's inequality
for Riemannian submersion between two Riemannian manifolds and their
application to real space form and complex space form. In section \ref{sec 4}%
, we give some basic definitions and notations related to Riemannian maps.
In section \ref{sec 3}, we establish we establish generalized Chen's
inequality for Riemannian maps between two Riemannian manifolds and their
application to real space form and complex space form.
\section{Riemannian submersion\label{sec 2}}

Let $\left( M_{1},g_{1}\right) $ and $\left( M_{2},g_{2}\right) $ be two
Riemannian manifolds of dimension $n$ and $m$, respectively. A surjective
smooth map $\pi :\left( M_{1},g_{1}\right) \rightarrow \left(
M_{2},g_{2}\right) $ is called a {\em Riemannian submersion} if its
differential map $\pi _{\ast p}:T_{p}M_{1}\rightarrow T_{\pi (p)}M_{2}$ is
surjective for all $p\in M_{1}$ and $\pi _{\ast p}$ preserves the length of
all horizontal vectors at $p$ \cite{Neill_66_MMJ}. The tangent vectors to
the fibers are vertical, while the orthogonal vectors to the fibers are
horizontal. Thus $TM_{1}$, the decomposition is a direct sum of two
distributions: the vertical distribution ${\cal V}=\ker \pi _{\ast }$ and
its orthogonal complement (known as the horizontal distribution) ${\cal H}%
=\left( \ker \pi _{\ast }\right) ^{\perp }$. Similarly, for each $p\in M_{1}$
the vertical and horizontal spaces in $T_{p}M_{1}$ are denoted by ${\cal V}%
_{p}=\left( \ker \pi _{\ast }\right) _{p}$ and ${\cal H}_{p}=\left( \ker \pi
_{\ast }\right) _{p}^{\perp }$, respectively. \newline
Let $\{V_{1},\ldots ,V_{r},U_{1},\ldots ,U_{s}\}$ be an orthonormal basis of 
$T_{p}M_{1}$ such that $\{V_{1},\ldots ,V_{r}\}$ and $\{U_{1},\ldots
,U_{s}\} $ are orthonormal bases of the vertical space ${\cal V}_{p}$ and
the horizontal space ${\cal H}_{p}$, respectively. Consider the planes $\Pi =%
{\rm span}\{V_{1},V_{2}\}$ and ${\Bbb P}={\rm span}\{U_{1},U_{2}\}$. We
shall use these bases throughout this paper.

\subsection{O'Neill tensors}

The geometry of Riemannian submersions is characterized by O'Neill's tensors 
{${\cal T}$} and {${\cal A}$} defined for vector fields $E$, $F$ on $M_{1}$
by 
\begin{eqnarray*}
{\cal T}\left( E,F\right) &=&{{\cal T}_{E}F=h\nabla _{vE}vF+v\nabla _{vE}hF},
\\
{\cal A}\left( E,F\right) &=&{\cal A}_{E}F=v{\nabla _{hE}hF+h\nabla _{hE}vF,}
\end{eqnarray*}%
where $\nabla $ is the Levi-Civita connection on $g_{1}$, $h$ and $v$ are
projection morphisms of $E$, $F\in TM_{1}$ to ${\cal H}$ and ${\cal V}$,
respectively. We also have%
\begin{eqnarray*}
{{\cal T}^{{\cal H}}}:{\cal V}\times {{\cal V}\rightarrow {\cal H}}, \\
{{\cal T}^{{\cal V}}}{:{\cal V}\times {{\cal H}\rightarrow {\cal V}}}, \\
{{\cal A}^{{\cal H}}}{:{\cal H}\times {{\cal V}\rightarrow {\cal H}}}, \\
{{\cal A}^{{\cal V}}}{:{\cal H}\times {{\cal H}\rightarrow {\cal V}}},
\end{eqnarray*}%
and 
\[
g_{1}\left( {{\cal T}^{{\cal H}}}\left( V_{i},V_{j}\right) ,U_{\alpha
}\right) =\left( {{\cal T}^{{\cal H}}}\right) _{ij}^{\alpha },\quad
g_{1}\left( {{\cal A}^{{\cal V}}}\left( U_{\alpha },U_{\beta }\right)
,V_{i}\right) =\left( {{\cal A}^{{\cal V}}}\right) _{\alpha \beta }^{i}, 
\]%
where $V_{i},V_{j}\in {\cal V}$, $U_{\alpha },U_{\beta }\in {\cal H}$. The
O'Neill tensors also satisfy: 
\[
{\cal A}_{U_{i}}^{{\cal V}}U_{j}=-{\cal A}_{U_{j}}^{{\cal V}}U_{i},\quad 
{\cal T}_{V_{i}}^{{\cal H}}V_{j}={\cal T}_{V_{j}}^{{\cal H}}V_{i}, 
\]%
and 
\[
g_{1}\left( {\cal T}_{E}F,G\right) =-g_{1}\left( F,{\cal T}_{E}G\right)
,\quad g_{1}\left( {\cal A}_{E}F,G\right) =-g_{1}\left( F,{\cal A}%
_{E}G\right) , 
\]%
where $V_{i},V_{j}\in {\cal V}$, $U_{i},U_{j}\in {\cal H}$ and $E$,
\thinspace $F$, $G\in TM_{1}$ \cite{Neill_66_MMJ}. The {\it mean curvature
vector field }$H${\it \ of the fibers} of $\pi $ is defined as \cite%
{Falcitelli_04_WS} 
\begin{equation}
H(p)=\frac{1}{r}N,\quad N=\sum_{i=1}^{r}{\cal T}^{{\cal H}}\left(
V_{i},V_{i}\right) .  \label{eq-P-(9)}
\end{equation}%
The horizontal divergence of any vector field $U$ on ${\cal H}$ is given by $%
\breve{\delta}\left( U\right) $ and defined by%
\[
\breve{\delta}\left( U\right) =\sum_{i=1}^{s}g\left( \nabla
_{U_{i}}U,U_{i}\right) . 
\]%
From \cite{Gulbahar_Meric_Kilic_17_KJM}, we have 
\begin{equation}
\breve{\delta}\left( N\right) =\sum_{j=1}^{r}\sum_{i=1}^{s}g\left( \left(
\nabla _{U_{i}}{\cal T}\right) \left( V_{j},V_{j}\right) ,U_{i}\right) .
\label{eq-P-(9.3.1)}
\end{equation}

\subsection{Relations between Riemann curvature tensor fields}

Let $R^{M_{1}}$, $R^{M_{2}}$, $R^{{\cal V}}$, and $R^{{\cal H}}$ denote the
Riemann curvature tensor fields corresponding to $M_{1}$, $M_{2}$, ${\cal V}$%
, and ${\cal H}$, respectively. Then, we have 
\begin{eqnarray}
R^{M_{1}}\left( F_{1},F_{2},F_{3},F_{4}\right) &=&R^{{\cal V}}\left(
F_{1},F_{2},F_{3},F_{4}\right) +g_{1}\left( {\cal T}_{F_{1}}F_{4},{\cal T}%
_{F_{2}}F_{3}\right)  \nonumber \\
&&-g_{1}\left( {\cal T}_{F_{2}}F_{4},{\cal T}_{F_{1}}F_{3}\right) ,
\label{eq-P-(10)}
\end{eqnarray}%
\begin{eqnarray}
R^{M_{1}}\left( X,Y,Z,W\right) &=&R^{{\cal H}}\left( X,Y,Z,W\right)
-2g_{1}\left( {\cal A}_{X}Y,{\cal A}_{Z}W\right)  \nonumber \\
&&+g_{1}\left( {\cal A}_{Y}Z,{\cal A}_{X}W\right) -g_{1}\left( {\cal A}_{X}Z,%
{\cal A}_{Y}W\right) ,  \label{eq-P-(11)}
\end{eqnarray}%
\begin{eqnarray}
R^{M_{1}}\left( X,F_{1},Y,F_{2}\right) &=&g_{1}\left( \left( \nabla _{X}%
{\cal T}\right) \left( F_{1},F_{2}\right) ,Y\right) +g_{1}\left( \left(
\nabla _{F_{1}}{\cal A}\right) \left( X,Y\right) ,F_{2}\right)  \nonumber \\
&&-g_{1}\left( {\cal T}_{F_{1}}X,{\cal T}_{F_{2}}Y\right) +g_{1}\left( {\cal %
A}_{Y}F_{2},{\cal A}_{X}F_{1}\right)  \label{eq-P-(12)}
\end{eqnarray}%
for all $X,Y,Z,W\in {\cal H}$\ and $F_{1},F_{2},F_{3},F_{4}\in {\cal V}$.
Here, $\nabla $ is the Levi-Civita connection with respect to the metric $%
g_{1}$ \cite{Falcitelli_04_WS,Neill_66_MMJ}.

\subsubsection{Some notations}

The scalar curvatures defined by

\begin{equation}
2\tau _{{\cal V}}^{{\cal V}}\left( p\right) =\sum_{i,j=1}^{r}R^{{\cal V}%
}\left( V_{i},V_{j},V_{j},V_{i}\right) ,\ 2\tau _{{\cal H}}^{{\cal H}}\left(
p\right) =\sum_{i,j=1}^{s}R^{{\cal V}}\left( U_{i},U_{j},U_{j},U_{i}\right)
\label{eq-P-(9.1)}
\end{equation}%
\begin{equation}
2\tau _{{\cal V}}^{{\cal V}}\left( \Pi _{j}\right) =\sum_{i,j=1}^{r_{j}}R^{%
{\cal V}}\left( V_{i},V_{j},V_{j},V_{i}\right) ,\ 2\tau _{{\cal H}}^{{\cal H}%
}\left( {\Bbb P}_{j}\right) =\sum_{i,j=1}^{s_{j}}R^{{\cal V}}\left(
U_{i},U_{j},U_{j},U_{i}\right)  \label{eq-P-(9.2)}
\end{equation}%
\begin{equation}
2\tau _{{\cal V}}^{M_{1}}(\Pi
_{j})=\sum\limits_{i,j=1}^{r_{j}}R^{M_{1}}\left(
V_{i},V_{j},V_{j},V_{i}\right) ,\ 2\tau _{{\cal H}}^{M_{1}}({\Bbb P}%
_{j})=\sum\limits_{i,j=1}^{s_{j}}R^{M_{1}}\left(
U_{i},U_{j},U_{j},U_{i}\right) ,  \label{eq-P-(9.2.0)}
\end{equation}%
\begin{equation}
2\tau _{{\cal V}}^{M_{1}}(p)=\sum\limits_{i,j=1}^{r}R^{M_{1}}\left(
V_{i},V_{j},V_{j},V_{i}\right) ,\ 2\tau _{{\cal H}}^{M_{1}}(p)=\sum%
\limits_{i,j=1}^{s}R^{M_{1}}\left( U_{i},U_{j},U_{j},U_{i}\right) ,
\label{eq-P-(9.2.1)}
\end{equation}%
\begin{equation}
\tau ^{M_{1}}(p)=\tau _{{\cal V}}^{M_{1}}(p)+\tau _{{\cal H}%
}^{M_{1}}(p)+\sum_{i=1}^{s}\sum_{j=1}^{r}R^{M_{1}}\left(
U_{i},V_{j},V_{j},U_{i}\right) .  \label{eq-scal-M1}
\end{equation}%
Also we have%
\begin{equation}
\left\Vert {\cal T}^{{\cal V}}\right\Vert
^{2}=\sum_{j=1}^{r}\sum_{i=1}^{s}g\left( {\cal T}_{V_{j}}^{{\cal V}}U_{i},%
{\cal T}_{V_{j}}^{{\cal V}}U_{i}\right) ,\ \left\Vert {\cal T}^{{\cal H}%
}\right\Vert ^{2}=\sum_{i,j=1}^{r}g\left( {\cal T}_{V_{i}}^{{\cal V}}V_{j},%
{\cal T}_{V_{i}}^{{\cal V}}V_{j}\right) ,  \label{eq-P-(14)}
\end{equation}%
\begin{equation}
\left\Vert {\cal A}^{{\cal V}}\right\Vert ^{2}=\sum_{i,j=1}^{s}g\left(
A_{U_{i}}^{{\cal V}}U_{j},A_{U_{i}}^{{\cal V}}U_{j}\right) ,\quad \left\Vert 
{\cal A}^{{\cal H}}\right\Vert ^{2}=\sum_{j=1}^{r}\sum_{i=1}^{s}g\left(
A_{U_{i}}^{{\cal H}}V_{j},A_{U_{i}}^{{\cal H}}V_{j}\right) ,
\label{eq-P-(14.1)}
\end{equation}%
\[
\left\Vert Q\right\Vert ^{2}=\sum\limits_{i=1}^{r}\Vert QV_{i}\Vert
^{2}=\sum\limits_{i,j=1}^{r}\left( g_{2}(QV_{i},V_{j})\right) ^{2},\
\left\Vert P\right\Vert ^{2}=\sum\limits_{i=1}^{s}\Vert PU_{i}\Vert
^{2}=\sum\limits_{i,j=1}^{s}\left( g_{2}(PU_{i},U_{j})\right) ^{2}, 
\]%
\[
\left\Vert P^{{\cal V}}\right\Vert ^{2}=\sum\limits_{i=1}^{r}\Vert
PV_{i}\Vert ^{2}=\sum\limits_{i=1}^{r}\sum\limits_{j=1}^{s}\left(
g_{2}(PV_{i},U_{j})\right) ^{2}. 
\]

\begin{definition}
The Kulkarni-Nomizu product $T_{1}\circledast T_{2}$ of $\left( 0,2\right) $%
-tensor fields $T_{1}$ and $T_{2}$ in a smooth manifold $M$ is a $\left(
0,4\right) $-tensor field defined by \cite{Tripathi_25_CMAMS} 
\begin{eqnarray}
\left( T_{1}\circledast T_{2}\right) \left( X,Y,Z,W\right) &=&T_{1}\left(
Y,Z\right) T_{2}\left( X,W\right) -T_{1}\left( X,Z\right) T_{2}\left(
Y,W\right)  \nonumber \\
&&+T_{2}\left( Y,Z\right) T_{1}\left( X,W\right) -T_{2}\left( X,Z\right)
T_{1}\left( Y,W\right)  \label{eq-KN-product}
\end{eqnarray}%
for all vector fields $X$, $Y$, $Z$, $W$ on $M$.
\end{definition}

\begin{definition}
The symmetric product of any two $\left( 0,2\right) $-tensor fields $T_{1}$
and $T_{2}$ is a $\left( 0,4\right) $-tensor field $T_{1}\circledcirc T_{2}$
defined by \cite{Tripathi_25_CMAMS} 
\begin{equation}
T_{1}\circledcirc T_{2}=T_{1}\circledast T_{2}+T_{2}\circledast T_{1}
\label{eq-KN-symm-product}
\end{equation}
\end{definition}

Now, let $\left( M,g\right) $ be an $n$-dimensional Riemannian manifold. Let 
$T$ be a Kulkarni-Nomizu tensor field so that it satisfies%
\begin{equation}
T\left( X,Y,Z,W\right) =-T\left( Y,X,Z,W\right) ,  \label{eq-KN-T-1}
\end{equation}%
\begin{equation}
T\left( X,Y,Z,W\right) =-T\left( X,Y,W,Z\right) ,  \label{eq-KN-T-2}
\end{equation}%
\begin{equation}
T\left( X,Y,Z,W\right) =T\left( Z,W,X,Y\right) ,  \label{eq-KN-T-3}
\end{equation}%
\begin{equation}
T\left( X,Y,Z,W\right) +T\left( Y,Z,X,W\right) +T\left( Z,X,Y,W\right) =0
\label{eq-KN-T-4}
\end{equation}%
\begin{equation}
T\left( X,Y,Z,W\right) +T\left( X,Z,W,Y\right) +T\left( X,W,Y,Z\right) =0
\label{eq-KN-T-5}
\end{equation}%
for all vector fields $X$, $Y$, $Z$ and $W$ on $M$. It is observed that if $%
T $ satisfies any two of the three conditions (\ref{eq-KN-T-1}), (\ref%
{eq-KN-T-2}), (\ref{eq-KN-T-3}) and any one of the two conditions (\ref%
{eq-KN-T-4}), (\ref{eq-KN-T-5}), then it also satisfies the remainning two
relations.

\begin{definition}
{\rm (\cite[Kobayashi and Nomizu, 1963, p. 209]{Kobayashi_Nomizu_1963},\cite[%
Takahashi, 1972]{Takahashi_72_KJSM})} A Riemannian manifold $\left(
M,g\right) $ with constant sectional curvature $c$ is called a real space
form, and its Riemann-Christoffel curvature tensor field $R^{M}$ is given by 
{\rm \cite[Tripathi, 2026, p. 29]{Tripathi_2026}} 
\begin{equation}
R^{M}=\frac{c}{2}\left( g\circledast g\right)  \label{eq-RSF}
\end{equation}
\end{definition}

\begin{definition}
{\rm \cite[Ogiue 1972]{Ogiue_72_JMSJ}} Let $M$ be an almost Hermitian
manifold with an almost Hermitian structure $\left( J,g\right) $. Then $M$
becomes a K\"{a}hler manifold if $\nabla J=0$. A k\"{a}hler manifold with
constant holomorphic sectional curvature $c$ is called a complex space form $%
M\left( c\right) $, and its Riemann-Christoffel curvature tensor field is
given by {\rm \cite[Tripathi, 2026, p. 79]{Tripathi_2026}} 
\begin{equation}
R^{M}=\frac{c}{8}\left( g\circledast g\right) +\frac{c}{4}\left\{ \frac{1}{2}%
\left( J^{b}\circledast J^{b}\right) -\left( J^{b}\circledcirc J^{b}\right)
\right\},  \label{eq-GCSF}
\end{equation}
where $J^{\flat}$ denotes the $(0,2)$-tensor field associated with the $%
(1,1) $-tensor field $J$, defined by 
\[
J^{\flat}(X,Y)=g(X,JY), 
\]
for all vector fields $X$, $Y$ on $M$. Moreover, for any vector field $X$ on 
$M$, we write 
\begin{equation}
JX=PX+QX,  \label{decompose_GCSF_RM}
\end{equation}%
where $PX\in {\cal H}$, $QX\in {\cal V}$.
\end{definition}

\begin{lemma}
{\rm \cite[Lemma 3.1, Chen 1993]{Chen_93_AM}} \label{Lemma 2}If $k>2$ and $%
a_{1},\ldots ,a_{k},b$ are real numbers such that 
\begin{equation}
\left( \sum_{i=1}^{k}a_{i}\right) ^{2}=\left( k-1\right) \left(
\sum_{i=1}^{k}a_{i}^{2}+b\right)  \label{eq-P-(13)}
\end{equation}%
then 
\[
2a_{1}a_{2}\geq b. 
\]%
The equality holds if and only if $a_{1}+a_{2}=a_{3}=\cdots =a_{k}$.
\end{lemma}

\section{Generalized Chen's invariant for Riemannian submersion along
vertical distribution \label{sec 3}}

\begin{theorem}
\label{Theorem 2} Let $\pi :\left( M_{1},g_{1}\right) \rightarrow \left(
M_{2},g_{2}\right) $ be a Riemannian submersion. Let $r_{1},\ldots ,r_{k}$
be intigers $\geq 2$ satisfying $r_{1}<r$, $r_{1}+\cdots +r_{k}\leq r$. For $%
p\in M_{1}$, let $L_{j}$ be an $r_{j}$-plane section of ${\cal V}_{p}$, $%
j=1,\ldots ,k$. Then we have%
\begin{equation}
\tau _{{\cal V}}^{\ker \pi _{\ast }}(p)-\sum_{j=1}^{k}\tau _{{\cal V}}^{\ker
\pi _{\ast }}\left( L_{j}\right) \geq \tau _{{\cal V}}^{M_{1}}(p)-%
\sum_{j=1}^{k}\tau _{{\cal V}}^{M_{1}}\left( L_{j}\right) -c\left(
r_{1},\ldots ,r_{k}\right) \left\Vert H\right\Vert ^{2},  \label{eq-GCFI}
\end{equation}%
for any $k$-tuple $\left( r_{1},\ldots ,r_{k}\right) \in S\left( r\right) $,
and equality holds at $p\in M_{1}$ if and only if there exist an orthonormal
basis $\{V_{1},\ldots ,V_{r}\}$ of $\left( \ker \pi _{\ast }\right) _{p}$,
and $\{h_{1},\ldots ,h_{s}\}$ be an orthonormal basis of the normal space $%
(\ker \pi _{\ast })_{p}^{\perp }$ such that%
\begin{eqnarray*}
\sum_{\alpha _{1}\in D_{1}}a_{\alpha _{1}} &=&\sum_{\alpha _{2}\in
D_{2}}a_{\alpha _{2}}=\cdots =\sum_{\alpha _{r}\in D_{r}}a_{\alpha
_{r}}=a_{_{\left( \sum_{j=1}^{k}r_{j}\right) +1}}=\cdots =a_{r}, \\
\left( {\cal T}^{{\cal H}}\right) _{\lambda \mu }^{1} &=&0,\quad \left(
\lambda ,\mu \neq \lambda \right) \notin D^{2},\quad \lambda ,\mu \in
\left\{ 1,\ldots ,r\right\} , \\
\left( {\cal T}^{{\cal H}}\right) _{\lambda \mu }^{\ell } &=&0,\quad \ell
=2,\ldots ,s,\quad \lambda ,\mu \in \left\{ 1,\ldots ,r\right\} , \\
\sum_{\alpha _{j}\in D_{j}}\left( {\cal T}^{{\cal H}}\right) _{\alpha
_{j}\alpha _{j}}^{\ell } &=&0,\quad \ell =2,\ldots ,s,\quad j=1,\ldots k.
\end{eqnarray*}
\end{theorem}

\begin{proof}
Let $\{V_{1},\ldots ,V_{r}\}$ be an orthonormal basis
of $\left( \ker \pi _{\ast }\right) _{p}$, and let $\{h_{1},\ldots ,h_{s}\}$
be an orthonormal basis of the normal space $(\ker \pi _{\ast })_{p}^{\perp }
$, chosen so that the mean curvature vector $H$ lies along the direction of $%
V_{1}$. For convenience, we set%
\begin{equation}
a_{i}=\left( {\cal T}^{{\cal H}}\right) _{ii}^{1}=g\left( \left( {\cal T}^{%
{\cal H}}\right) \left( V_{i},V_{i}\right) ,h_{1}\right) ,\quad i=1,\ldots
,r,  \label{eq-GCFI-(1.01)}
\end{equation}%
Define index sets 
\begin{eqnarray*}
D_{1} &=&\left\{ 1,\ldots ,r_{1}\right\} , \\
D_{2} &=&\left\{ r_{1}+1,\ldots ,r_{1}+r_{2}\right\}  \\
&&\vdots  \\
D_{k} &=&\left\{ r_{1}+\cdots +r_{k-1}+1,\ldots ,r_{1}+\cdots
+r_{k-1}+r_{k}\right\} 
\end{eqnarray*}%
Let $L_{1},\ldots ,L_{k}$ be mutually orthogonal subspaces of ${\cal V}_{p}$%
, with $\dim L_{j}=r_{j}$, defined by%
\[
L_{j}={\rm span}\{e_{r_{1}+\cdots +r_{j-1}+1},...,e_{r_{1}+\cdots
+r_{j-1}+r_{j}}\},\quad j=1,...,k.
\]%
From (\ref{eq-P-(9)}), (\ref{eq-P-(10)}) and (\ref{eq-P-(14)}), we obtain 
\begin{equation}
r^{2}\left\Vert H\right\Vert ^{2}=(r+k-\sum_{j=1}^{k}r_{j})\{2\tau _{{\cal V}%
}^{M_{1}}(p)-2\tau _{{\cal V}}^{\ker \pi _{\ast
}}(p)-2c(r_{1},...,r_{k})\left\Vert H\right\Vert ^{2}+\left\Vert \left( 
{\cal T}^{{\cal H}}\right) \right\Vert ^{2}\},  \label{eq-rh2}
\end{equation}%
equation (\ref{eq-rh2}) can be written as 
\begin{equation}
r^{2}\left\Vert H\right\Vert ^{2}=(\left\Vert \left( {\cal T}^{{\cal H}%
}\right) \right\Vert ^{2}+\eta )\gamma ,  \label{GCFI-(1.1)}
\end{equation}%
where%
\begin{eqnarray}
\eta  &=&2\tau _{{\cal V}}^{M_{1}}(p)-2\tau _{{\cal V}}^{\ker \pi _{\ast
}}(p)-2c(r_{1},...,r_{k})\left\Vert H\right\Vert ^{2},  \label{eq-eta} \\
\gamma  &=&r+k-\sum_{j=1}^{k}r_{j},  \label{eq-gamma} \\
2c &=&\frac{r^{2}(r+k-1-\sum_{j=1}^{k}r_{j})}{(r+k-\sum_{j=1}^{k}r_{j})}.
\label{eq-GCFI-(1.1b)}
\end{eqnarray}%
Since mean curvature vector $H$ lies in the direction of $h_{1}$, then from (%
\ref{GCFI-(1.1)}), we obtain%
\begin{equation}
\left( \sum_{i=1}^{r}\left( {\cal T}^{{\cal H}}\right) _{ii}^{1}\right)
^{2}=\gamma \left[ \sum_{i=1}^{r}(\left( {\cal T}^{{\cal H}}\right)
_{ii}^{1})^{2}+\sum_{i\neq j=1}^{r}(\left( {\cal T}^{{\cal H}}\right)
_{ij}^{1})^{2}+\sum_{\ell =2}^{s}\sum_{i,j=1}^{r}(\left( {\cal T}^{{\cal H}%
}\right) _{ij}^{\ell })^{2}+\eta \right] .  \label{eq-GCFI-(1.2)}
\end{equation}%
From (\ref{eq-GCFI-(1.2)}) and (\ref{eq-GCFI-(1.01)}), we obtain%
\begin{equation}
\left( \sum_{i=1}^{r}a_{i}\right) ^{2}=\gamma \left( \sum_{i=1}^{r}\left(
a_{i}\right) ^{2}+\sum_{i\neq j=1}^{r}\left( \left( {\cal T}^{{\cal H}%
}\right) _{ij}^{1}\right) ^{2}+\sum_{\ell =1}^{s}\sum_{i=1}^{r}\left( \left( 
{\cal T}^{{\cal H}}\right) _{ii}^{\ell }\right) ^{2}+\eta \right) ,
\label{eq-ai}
\end{equation}%
equation (\ref{eq-ai}) can be written as%
\begin{eqnarray}
\left( \sum_{i=1}^{r}a_{i}\right) ^{2} &=&\gamma \left\{
\sum_{i=1}^{r}\left( a_{i}\right) ^{2}+\sum_{i\neq j=1}^{r}\left( \left( 
{\cal T}^{{\cal H}}\right) _{ij}^{1}\right) ^{2}+\sum_{\ell
=1}^{s}\sum_{i=1}^{r}\left( \left( {\cal T}^{{\cal H}}\right) _{ii}^{\ell
}\right) ^{2}+\eta \right.   \nonumber \\
&&\left. +\sum_{2\leq \alpha _{1}\neq \beta _{1}\leq r_{1}}a_{\alpha
_{1}}a_{\beta _{1}}-\sum_{2\leq \alpha _{1}\neq \beta _{1}\leq
r_{1}}a_{\alpha _{1}}a_{\beta _{1}}+\sum_{\alpha _{2}\neq \beta
_{2}}a_{\alpha _{2}}a_{\beta _{2}}\right.   \nonumber \\
&&\left. -\sum_{\alpha _{2}\neq \beta _{2}}a_{\alpha _{2}}a_{\beta
_{2}}+\cdots +\sum_{\alpha _{k}\neq \beta _{k}}a_{\alpha _{k}}a_{\beta
_{k}}-\sum_{\alpha _{k}\neq \beta _{k}}a_{\alpha _{k}}a_{\beta _{k}}\right\}
,  \label{eq-(1.2aa)}
\end{eqnarray}%
for $\alpha _{2},\beta _{2}\in D_{2},\ldots ,\alpha _{k},\beta _{k}\in D_{k}$%
. Denoting

\begin{eqnarray*}
b_{1} &=&a_{1},\ b_{2}=a_{2}+\cdots +a_{r_{1}},\ b_{3}=a_{r_{1}+1}+\cdots
+a_{r_{1}+r_{2}},\cdots , \\
b_{k+1} &=&a_{r_{k-1}+1}+\cdots +a_{r_{1}+r_{2}+\cdots +r_{k}},\quad
b_{k+2}=a_{\left( \sum_{j=1}^{k}r_{j}\right) +1},\cdots , \\
b_{\gamma +1} &=&b_{r+k-\sum_{j=1}^{k}r_{j}+1}=a_{\left(
\sum_{j=1}^{k}r_{j}\right) +r-\sum_{j=1}^{k}r_{j}}=a_{r},
\end{eqnarray*}%
equation (\ref{eq-(1.2aa)}) becomes%
\begin{eqnarray*}
\left( \sum_{i=1}^{\gamma +1}b_{i}\right) ^{2} &=&\gamma \left\{
\sum_{i=1}^{r}\left( b_{i}\right) ^{2}+\sum_{i\neq j=1}^{r}\left( \left( 
{\cal T}^{{\cal H}}\right) _{ij}^{1}\right) ^{2}+\sum_{\ell
=2}^{s}\sum_{i=1}^{r}\left( \left( {\cal T}^{{\cal H}}\right) _{ii}^{\ell
}\right) ^{2}+\eta \right.  \\
&&\left. -\sum_{2\leq \alpha _{1}\neq \beta _{1}\leq r_{1}}a_{\alpha
_{1}}a_{\beta _{1}}-\sum_{\alpha _{2}\neq \beta _{2}}a_{\alpha _{2}}a_{\beta
_{2}}-\cdots -\sum_{\alpha _{k}\neq \beta _{k}}a_{\alpha _{k}}a_{\beta
_{k}}\right\} ,
\end{eqnarray*}%
for $\alpha _{2},\beta _{2}\in D_{2},\ldots ,\alpha _{k},\beta _{k}\in D_{k}$%
. Applying Lemma~\ref{Lemma 2} to the above relation, we obtain%
\begin{eqnarray}
2b_{1}b_{2} &\geq &\left\{ \sum_{i\neq j=1}^{r}\left( \left( {\cal T}^{{\cal %
H}}\right) _{ij}^{1}\right) ^{2}+\sum_{\ell =2}^{s}\sum_{i=1}^{r}\left(
\left( {\cal T}^{{\cal H}}\right) _{ii}^{\ell }\right) ^{2}+\eta \right.  
\nonumber \\
&&\left. -\sum_{2\leq \alpha _{1}\neq \beta _{1}\leq r_{1}}a_{\alpha
_{1}}a_{\beta _{1}}-\sum_{\alpha _{2}\neq \beta _{2}}a_{\alpha _{2}}a_{\beta
_{2}}-\cdots -\sum_{\alpha _{k}\neq \beta _{k}}a_{\alpha _{k}}a_{\beta
_{k}}\right\} ,  \label{eq-GCFI-(1.2a)}
\end{eqnarray}%
which implies 
\begin{eqnarray}
&&\sum_{\alpha _{1}<\beta _{1}}a_{\alpha _{1}}a_{\beta _{1}}+\sum_{\alpha
_{2}<\beta _{2}}a_{\alpha _{2}}a_{\beta _{2}}+\cdots +\sum_{\alpha
_{k}<\beta _{k}}a_{\alpha _{k}}a_{\beta _{k}}  \nonumber \\
&&\quad \quad \quad \quad \quad \left. \geq \frac{\eta }{2}+\frac{1}{2}%
\sum_{i\neq j=1}^{r}\left( \left( {\cal T}^{{\cal H}}\right)
_{ij}^{1}\right) ^{2}+\frac{1}{2}\sum_{\ell =2}^{s}\sum_{i=1}^{r}\left(
\left( {\cal T}^{{\cal H}}\right) _{ii}^{\ell }\right) ^{2}\right. 
\label{eq-GCFI-(1.3)}
\end{eqnarray}%
for $\alpha _{j},\beta _{j}\in D_{j}$, $j=1,\ldots ,k$. From (\ref{eq-P-(10)}%
), for $\alpha _{j},\beta _{j}\in D_{j}$, $j=1,\ldots ,k$, we have%
\[
\sum_{\alpha _{j}<\beta _{j}}K_{\alpha _{j}\beta _{j}}^{M_{1}}=\sum_{\alpha
_{j}<\beta _{j}}K_{\alpha _{j}\beta _{j}}^{\ker \pi _{\ast }}+\sum_{\ell
=1}^{s}\sum_{\alpha _{j}<\beta _{j}}\left( \left( {\cal T}^{{\cal H}}\right)
_{\alpha _{j}\alpha _{j}}^{\ell }\left( {\cal T}^{{\cal H}}\right) _{\beta
_{j}\beta _{j}}^{\ell }-\left( \left( {\cal T}^{{\cal H}}\right) _{\alpha
_{j}\beta _{j}}^{\ell }\right) ^{2}\right) ,
\]%
\begin{eqnarray}
\sum_{j=1}^{k}\tau _{{\cal V}}^{M_{1}}\left( L_{j}\right) 
&=&\sum_{j=1}^{k}\tau _{{\cal V}}^{\ker \pi _{\ast }}\left( L_{j}\right)
+\sum_{j=1}^{k}\sum_{\alpha _{j}<\beta _{j}}\left( \left( {\cal T}^{{\cal H}%
}\right) _{\alpha _{j}\alpha _{j}}^{1}\left( {\cal T}^{{\cal H}}\right)
_{\beta _{j}\beta _{j}}^{1}-\left( \left( {\cal T}^{{\cal H}}\right)
_{\alpha _{j}\beta _{j}}^{1}\right) ^{2}\right)   \nonumber \\
&&+\sum_{j=1}^{k}\sum_{\ell =2}^{s}\sum_{\alpha _{j}<\beta _{j}}\left(
\left( {\cal T}^{{\cal H}}\right) _{\alpha _{j}\alpha _{j}}^{\ell }\left( 
{\cal T}^{{\cal H}}\right) _{\beta _{j}\beta _{j}}^{\ell }-\left( \left( 
{\cal T}^{{\cal H}}\right) _{\alpha _{j}\beta _{j}}^{\ell }\right)
^{2}\right) .  \label{eq-GCFI-(1.4)}
\end{eqnarray}%
Define%
\begin{eqnarray*}
D &=&D_{1}\cup D_{2}\cup \cdots \cup D_{k} \\
D^{2} &=&D_{1}\times D_{1}\cup D_{2}\times D_{2}\cup \cdots \cup D_{k}\times
D_{k}.
\end{eqnarray*}%
Substituting the value from (\ref{eq-GCFI-(1.3)}) into (\ref{eq-GCFI-(1.4)}%
), we get%
\begin{eqnarray}
\sum_{j=1}^{k}\tau _{{\cal V}}^{M_{1}}\left( L_{j}\right)  &\geq
&\sum_{j=1}^{k}\tau _{{\cal V}}^{\ker \pi _{\ast }}\left( L_{j}\right) +%
\frac{\eta }{2}+\frac{1}{2}\sum_{i\neq j=1}^{r}\left( \left( {\cal T}^{{\cal %
H}}\right) _{ij}^{1}\right) ^{2}+\frac{1}{2}\sum_{\ell
=2}^{s}\sum_{i=1}^{r}\left( \left( {\cal T}^{{\cal H}}\right) _{ii}^{\ell
}\right) ^{2}  \nonumber \\
&&-\sum_{j=1}^{k}\sum_{\alpha _{j}<\beta _{j}}\left( \left( {\cal T}^{{\cal H%
}}\right) _{\alpha _{j}\beta _{j}}^{1}\right) ^{2}+\sum_{j=1}^{k}\sum_{\ell
=2}^{s}\sum_{\alpha _{j}<\beta _{j}}\left( {\cal T}^{{\cal H}}\right)
_{\alpha _{j}\alpha _{j}}^{\ell }\left( {\cal T}^{{\cal H}}\right) _{\beta
_{j}\beta _{j}}^{\ell }  \nonumber \\
&&-\sum_{j=1}^{k}\sum_{\ell =2}^{s}\sum_{\alpha _{j}<\beta _{j}}\left(
\left( {\cal T}^{{\cal H}}\right) _{\alpha _{j}\beta _{j}}^{\ell }\right)
^{2}.  \label{eq-term-3-5}
\end{eqnarray}%
Form (\ref{eq-term-3-5}), we get 
\begin{eqnarray}
\sum_{j=1}^{k}\tau _{{\cal V}}^{M_{1}}\left( L_{j}\right)  &\geq
&\sum_{j=1}^{k}\tau _{{\cal V}}^{\ker \pi _{\ast }}\left( L_{j}\right) +%
\frac{\eta }{2}+\frac{1}{2}\sum_{\left( \lambda ,\mu \neq \lambda \right)
\notin D^{2}}\left( \left( {\cal T}^{{\cal H}}\right) _{\lambda \mu
}^{1}\right) ^{2}+\frac{1}{2}\sum_{\ell =2}^{s}\sum_{i=1}^{r}\left( \left( 
{\cal T}^{{\cal H}}\right) _{ii}^{\ell }\right) ^{2}  \nonumber \\
&&+\sum_{j=1}^{k}\sum_{\ell =2}^{s}\sum_{\alpha _{j}<\beta _{j}}\left( {\cal %
T}^{{\cal H}}\right) _{\alpha _{j}\alpha _{j}}^{\ell }\left( {\cal T}^{{\cal %
H}}\right) _{\beta _{j}\beta _{j}}^{\ell }-\sum_{j=1}^{k}\sum_{\ell
=2}^{s}\sum_{\alpha _{j}<\beta _{j}}\left( \left( {\cal T}^{{\cal H}}\right)
_{\alpha _{j}\beta _{j}}^{\ell }\right) ^{2}.  \label{eq-binomial}
\end{eqnarray}%
By using the binomial expansion in (\ref{eq-binomial}), we obtain 
\begin{eqnarray}
\sum_{j=1}^{k}\tau _{{\cal V}}^{M_{1}}\left( L_{j}\right)  &\geq
&\sum_{j=1}^{k}\tau _{{\cal V}}^{\ker \pi _{\ast }}\left( L_{j}\right) +%
\frac{\eta }{2}+\frac{1}{2}\sum_{\left( \lambda ,\mu \neq \lambda \right)
\notin D^{2}}\left( \left( {\cal T}^{{\cal H}}\right) _{\lambda \mu
}^{1}\right) ^{2}+\frac{1}{2}\sum_{\ell =2}^{s}\sum_{\left( \lambda ,\mu
\right) \notin D^{2}}\left( \left( {\cal T}^{{\cal H}}\right) _{\lambda \mu
}^{\ell }\right) ^{2}  \nonumber \\
&&+\frac{1}{2}\sum_{j=1}^{k}\sum_{\ell =2}^{s}\left( \sum_{\alpha _{j}\in
D_{j}}\left( {\cal T}^{{\cal H}}\right) _{\alpha _{j}\alpha _{j}}^{\ell
}\right) ^{2}.  \label{eq-afterbin}
\end{eqnarray}%
Consequently, From (\ref{eq-afterbin}), we get 
\begin{equation}
\sum_{j=1}^{k}\tau _{{\cal V}}^{M_{1}}\left( L_{j}\right) \geq
\sum_{j=1}^{k}\tau _{{\cal V}}^{\ker \pi _{\ast }}\left( L_{j}\right) +\frac{%
\eta }{2},  \label{eq-GCFI-(1.5)}
\end{equation}%
From (\ref{eq-eta}) and (\ref{eq-GCFI-(1.5)}), we obtain 
\[
\tau _{{\cal V}}^{\ker \pi _{\ast }}(p)-\sum_{j=1}^{k}\tau _{{\cal V}}^{\ker
\pi _{\ast }}\left( L_{j}\right) \geq \tau _{{\cal V}}^{M_{1}}(p)-%
\sum_{j=1}^{k}\tau _{{\cal V}}^{M_{1}}\left( L_{j}\right) -c\left(
r_{1},\ldots ,r_{k}\right) \left\Vert H\right\Vert ^{2}.
\]%
The equality case in (\ref{eq-GCFI}) holds if and only if equalities in (\ref%
{eq-GCFI-(1.2a)}) and (\ref{eq-GCFI-(1.5)}) are satisfied. According to
Lemma \ref{Lemma 2}, equality in (\ref{eq-GCFI-(1.2a)}) holds if and only if%
\[
b_{1}+b_{2}=b_{3}=\cdots =b_{k+1}=\cdots =b_{\gamma +1}=a_{r},
\]%
that is,%
\begin{equation}
\sum_{\alpha _{1}\in D_{1}}a_{\alpha _{1}}=\sum_{\alpha _{2}\in
D_{2}}a_{\alpha _{2}}=\cdots =\sum_{\alpha _{r}\in D_{r}}a_{\alpha
_{r}}=a_{_{\left( \sum_{j=1}^{k}r_{j}\right) +1}}=\cdots =a_{r},
\label{eq-GCFI-(1.6)}
\end{equation}%
equality in (\ref{eq-GCFI-(1.5)}) holds if and only if%
\begin{eqnarray}
\left( {\cal T}^{{\cal H}}\right) _{\lambda \mu }^{1} &=&0,\quad \left(
\lambda ,\mu \neq \lambda \right) \notin D^{2},\quad \lambda ,\mu \in
\left\{ 1,\ldots ,r\right\}   \nonumber \\
\left( {\cal T}^{{\cal H}}\right) _{\lambda \mu }^{\ell } &=&0,\quad \ell
=2,\ldots ,s,\quad \lambda ,\mu \in \left\{ 1,\ldots ,r\right\}   \nonumber
\\
\sum_{\alpha _{j}\in D_{j}}\left( {\cal T}^{{\cal H}}\right) _{\alpha
_{j}\alpha _{j}}^{\ell } &=&0,\quad \ell =2,\ldots ,s,\quad j=1,\ldots k.
\label{eq-GCFI-(1.7)}
\end{eqnarray}%
Which complete the proof.
\end{proof} 
\begin{definition}
Let $\pi :(M_{1},g_{1})\rightarrow (M_{2},g_{2})$ be a Riemannian submersion
between Riemannian manifolds with $\dim M_{1}=n$ and $\dim M_{2}=m$. For
each $\left( r_{1}.\ldots ,r_{k}\right) \in S\left( r\right) $, we define $%
\delta \left( r_{1}.\ldots ,r_{k}\right) \left( p\right) $ and $\hat{\delta}%
\left( r_{1}.\ldots ,r_{k}\right) \left( p\right) $ by%
\begin{eqnarray}
\delta _{{\cal V}}^{\ker \pi _{\ast }}\left( r_{1}.\ldots ,r_{k}\right)
\left( p\right)  &=&\tau _{{\cal V}}^{\ker \pi _{\ast }}(p)-\inf \left\{
\sum_{j=1}^{k}\tau _{{\cal V}}^{\ker \pi _{\ast }}\left( L_{j}\right)
\right\}   \nonumber \\
\delta _{{\cal V}}^{M_{1}}\left( r_{1}.\ldots ,r_{k}\right) \left( p\right) 
&=&\tau _{{\cal V}}^{M_{1}}(p)-\inf \left\{ \sum_{j=1}^{k}\tau _{{\cal V}%
}^{M_{1}}\left( L_{j}\right) \right\}   \label{eq-P-D-(14)}
\end{eqnarray}%
\begin{eqnarray}
\hat{\delta}_{{\cal V}}^{\ker \pi _{\ast }}\left( r_{1}.\ldots ,r_{k}\right)
\left( p\right)  &=&\tau _{{\cal V}}^{\ker \pi _{\ast }}(p)-\sup \left\{
\sum_{j=1}^{k}\tau _{{\cal V}}^{\ker \pi _{\ast }}\left( L_{j}\right)
\right\}   \nonumber \\
\hat{\delta}_{{\cal V}}^{M_{1}}\left( r_{1}.\ldots ,r_{k}\right) \left(
p\right)  &=&\tau _{{\cal V}}^{M_{1}}(p)-\sup \left\{ \sum_{j=1}^{k}\tau _{%
{\cal V}}^{M_{1}}\left( L_{j}\right) \right\}   \label{eq-P-D-(14.1)}
\end{eqnarray}%
where $L_{1},\ldots ,L_{k}$ run over all rmutually orthogonal subspaces of $%
{\cal V}_{p}$.
\end{definition}

In the following theorem, using Theorem \ref{Theorem 2}, we provide an
inequality for $\hat{\delta}_{{\cal V}}^{\ker \pi _{\ast }}\left(
r_{1}.\ldots ,r_{k}\right) \left( p\right) $ and $\hat{\delta}_{{\cal V}%
}^{M_{1}}\left( r_{1},\ldots ,r_{k}\right) \left( p\right) $.

\begin{theorem}
Let $\pi :\left( M_{1},g_{1}\right) \rightarrow \left( M_{2},g_{2}\right) $
be a Riemannian submersion. Let $r_{1},\ldots ,r_{k}$ be intigers $\geq 2$
satisfying $r_{1}<r$, $r_{1}+\cdots +r_{k}\leq r$. For $p\in M_{1}$, let $%
L_{j}$ be an $r_{j}$-plane section of ${\cal V}_{p}$, $j=1,\ldots ,k$. Then
we have%
\begin{equation}
\hat{\delta}_{{\cal V}}^{\ker \pi _{\ast }}\left( r_{1}.\ldots ,r_{k}\right)
\left( p\right) \geq \hat{\delta}_{{\cal V}}^{M_{1}}\left( r_{1}.\ldots
,r_{k}\right) \left( p\right) -c\left( r_{1},\ldots ,r_{k}\right) \left\Vert
H\right\Vert ^{2},  \label{eq-GCDI}
\end{equation}%
for any $k$-tuple $\left( r_{1},\ldots ,r_{k}\right) \in S\left( r\right) $,
and equality holds at $p\in M_{1}$ if and only if there exist an orthonormal
basis $\{V_{1},\ldots ,V_{r}\}$ of $\left( \ker \pi _{\ast }\right) _{p}$,
and $\{h_{1},\ldots ,h_{s}\}$ be an orthonormal basis of the normal space $%
(\ker \pi _{\ast })_{p}^{\perp }$ such that

\begin{enumerate}
\item equality in {\rm (\ref{eq-GCDI})} holds at a point $p\in M_{1}$ if and
only if it follows equality cases of Theorem {\rm \ref{Theorem 2}} and for
any $k$ mutually orthogonal subspaces $L_{1},\ldots ,L_{k}$ of ${\cal V}_{p}$
satisfying 
\begin{equation}
\hat{\delta}_{{\cal V}}^{\ker \pi _{\ast }}\left( r_{1}.\ldots ,r_{k}\right)
\left( p\right) =\tau _{{\cal V}}^{\ker \pi _{\ast }}(p)-\sum_{j=1}^{k}\tau
_{{\cal V}}^{\ker \pi _{\ast }}\left( L_{j}\right)   \label{eq-GCDI-(01)}
\end{equation}%
we have%
\begin{equation}
\hat{\delta}_{{\cal V}}^{M_{1}}\left( r_{1}.\ldots ,r_{k}\right) \left(
p\right) =\tau _{{\cal V}}^{M_{1}}(p)-\sum_{j=1}^{k}\tau _{{\cal V}%
}^{M_{1}}\left( L_{j}\right)   \label{eq-GCDI-(02)}
\end{equation}
\end{enumerate}
\end{theorem}

\begin{proof}
From (\ref{eq-GCFI}), we have%
\begin{equation}
\tau _{{\cal V}}^{\ker \pi _{\ast }}(p)-\sum_{j=1}^{k}\tau _{{\cal V}}^{\ker
\pi _{\ast }}\left( L_{j}\right) \geq \inf \left\{ \tau _{{\cal V}%
}^{M_{1}}(p)-\sum_{j=1}^{k}\tau _{{\cal V}}^{M_{1}}\left( L_{j}\right)
-c\left( r_{1},\ldots ,r_{k}\right) \left\Vert H\right\Vert ^{2}\right\} 
\label{eq-GCDI-(1)}
\end{equation}%
that is $\tau _{{\cal V}}^{M_{1}}(p)-\sum_{j=1}^{k}\tau _{{\cal V}%
}^{M_{1}}\left( L_{j}\right) -c\left( r_{1},\ldots ,r_{k}\right) \left\Vert
H\right\Vert ^{2}$ is an lower bound for $\tau _{{\cal V}}^{\ker \pi _{\ast
}}(p)-\sum_{j=1}^{k}\tau _{{\cal V}}^{\ker \pi _{\ast }}\left( L_{j}\right) $%
. Hence, from (\ref{eq-GCDI-(1)}), we get%
\begin{equation}
\inf \left\{ \tau _{{\cal V}}^{\ker \pi _{\ast }}(p)-\sum_{j=1}^{k}\tau _{%
{\cal V}}^{\ker \pi _{\ast }}\left( L_{j}\right) \right\} \geq \inf \left\{
\tau _{{\cal V}}^{M_{1}}(p)-\sum_{j=1}^{k}\tau _{{\cal V}}^{M_{1}}\left(
L_{j}\right) -c\left( r_{1},\ldots ,r_{k}\right) \left\Vert H\right\Vert
^{2}\right\} .  \label{eq-GCDI-(2)}
\end{equation}%
Since for a fixed $T_{p}M$ we have $\tau _{{\cal V}}^{\ker \pi _{\ast }}(p)$%
, $\tau _{{\cal V}}^{M_{1}}(p)$ and $c\left( r_{1},\ldots ,r_{k}\right)
\left\Vert H\right\Vert ^{2}$ are real constant. Thus from (\ref{eq-GCDI-(2)}%
), we have%
\[
\tau _{{\cal V}}^{\ker \pi _{\ast }}(p)-\sup \left\{ \sum_{j=1}^{k}\tau _{%
{\cal V}}^{\ker \pi _{\ast }}\left( L_{j}\right) \right\} \geq \tau _{{\cal V%
}}^{M_{1}}(p)-\sup \left\{ \sum_{j=1}^{k}\tau _{{\cal V}}^{M_{1}}\left(
L_{j}\right) \right\} -c\left( r_{1},\ldots ,r_{k}\right) \left\Vert
H\right\Vert ^{2}.
\]%
which in view of (\ref{eq-P-D-(14.1)}), gives (\ref{eq-GCDI}). Equality in (%
\ref{eq-GCDI}) holds at a point $p\in M_{1}$ if and only if it follows
equality cases of Theorem \ref{Theorem 2} and for any $k$ mutually
orthogonal subspaces $L_{1},\ldots ,L_{k}$ of ${\cal V}_{p}$ satisfying (\ref%
{eq-GCDI-(01)}), we have (\ref{eq-GCDI-(02)}).
\end{proof}

\begin{theorem}
Let $\pi :\left( M_{1}\left( c\right) ,g_{1}\right) \rightarrow \left(
M_{2},g_{2}\right) $ be a Riemannian submersion from real space form onto a
Riemannian manifold. Let $r_{1},\ldots ,r_{k}$ be intigers $\geq 2$
satisfying $r_{1}<r$, $r_{1}+\cdots +r_{k}\leq r$. For $p\in M_{1}$, let $%
L_{j}$ be an $r_{j}$-plane section of $\left( \ker \pi _{\ast }\right) _{p}$%
, $j=1,\ldots ,k$. Then we have%
\begin{equation}
\tau _{{\cal V}}^{\ker \pi _{\ast }}(p)-\sum_{j=1}^{k}\tau _{{\cal V}}^{\ker
\pi _{\ast }}(L_{j})\geq \frac{1}{2}c\left\{ r\left( r-1\right)
-\sum_{j=1}^{k}r_{j}\left( r_{j}-1\right) \right\} -c(r_{1},\ldots
,r_{k})\left\Vert H\right\Vert ^{2}.  \label{eq-GRSF-(1)}
\end{equation}%
Equality case follows Theorem {\rm \ref{Theorem 2}}.
\end{theorem}

\begin{proof}
By (\ref{eq-RSF}) and (\ref{eq-P-(9.2.1)}), we get%
\begin{equation}
\sum_{j=1}^{k}\tau _{{\cal V}}^{M_{1}}\left( L_{j}\right) =\frac{c}{2}%
\sum_{j=1}^{k}r_{j}\left( r_{j}-1\right) ,\quad \tau _{{\cal V}}^{M_{1}}(p)=%
\frac{c}{2}r\left( r-1\right) .  \label{eq-GRSF-(1.1)}
\end{equation}%
In view of (\ref{eq-GRSF-(1.1)})and (\ref{eq-GCFI}), we get (\ref%
{eq-GRSF-(1)}).
\end{proof}

\begin{theorem}
Let $\pi :\left( M_{1}\left( c\right) ,g_{1}\right) \rightarrow \left(
M_{2},g_{2}\right) $ be a Riemannian submersion from complex space form onto
a Riemannian manifold. Let $r_{1},\ldots ,r_{k}$ be intigers $\geq 2$
satisfying $r_{1}<r$, $r_{1}+\cdots +r_{k}\leq r$. For $p\in M_{1}$, let $%
L_{j}$ be an $r_{j}$-plane section of $\left( \ker \pi _{\ast }\right) _{p}$%
, $j=1,\ldots ,k$. Then we have%
\begin{eqnarray}
\tau _{{\cal V}}^{\ker \pi _{\ast }}(p)-\sum_{j=1}^{k}\tau _{{\cal V}}^{\ker
\pi _{\ast }}(L_{j}) &\geq &\frac{c}{8}\left\{ r\left( r-1\right)
-\sum_{j=1}^{k}r_{j}\left( r_{j}-1\right) \right\}   \nonumber \\
&&+\frac{3c}{8}\left\{ \left\Vert Q\right\Vert ^{2}-2\sum_{j=1}^{k}\Psi
\left( L_{j}\right) \right\} -c(r_{1},\ldots ,r_{k})\left\Vert H\right\Vert
^{2},  \label{eq-GCFGCSF}
\end{eqnarray}%
where $\Psi \left( L_{j}\right) =\sum_{1\leq i<j\leq r_{j}}g\left(
QV_{i},V_{j}\right) ^{2}$. Equality case follows Theorem {\rm \ref{Theorem 2}%
}.
\end{theorem}

\begin{proof}
 By (\ref{eq-GCSF}) and (\ref{eq-P-(9.2.1)}), we get%
\begin{equation}
\sum_{j=1}^{k}\tau _{{\cal V}}^{M_{1}}\left( L_{j}\right) =\frac{c}{8}%
\left\{ \sum_{j=1}^{k}r_{j}\left( r_{j}-1\right) +6\sum_{j=1}^{k}\Psi \left(
L_{j}\right) \right\} ,  \label{eq-GCFGCSF-(1.1)}
\end{equation}%
\begin{eqnarray}
\tau _{{\cal V}}^{M_{1}}(p) &=&\frac{c}{8}\left\{ r\left( r-1\right)
+3\sum_{i,j=1}^{r}g\left( QV_{i},V_{j}\right) ^{2}\right\}  \nonumber \\
&=&\frac{c}{8}\left\{ r\left( r-1\right) +3\left\Vert Q\right\Vert
^{2}\right\} .  \label{eq-GCFGCSF-(1.2)}
\end{eqnarray}%
In view of (\ref{eq-GCFGCSF-(1.1)}), (\ref{eq-GCFGCSF-(1.2)}) and (\ref%
{eq-GCFI}), we get (\ref{eq-GCFGCSF}).
\end{proof}

\section{Riemannian map \label{sec 4}}

In this section, we give some basic concepts to prove the main inequalities.%
\newline
Let $\pi: (M^m, g_1) \to (N^n, g_2)$ be a smooth map between Riemannian
manifolds with $0 < {\text{\normalfont rank }} \pi < \min\{m,n\}$, and let $%
\pi_{\ast p}: T_p M \to T_{\pi(p)} N$ be its differential map at $p$.
Denoting the kernel space of $\pi_\ast$ at $p \in M$ by ${\cal V}_p = (\ker
\pi_{\ast p})$ and its orthogonal complementary space in the tangent space $%
T_p M$ by ${\cal H}_{p}= (\ker \pi_{\ast p})^\perp$, we have 
\[
T_{p}M = {\cal V}_p \oplus {\cal H}_p. 
\]
Similarly, we have 
\[
T_{\pi(p)} N = {\cal R}_p \oplus {\cal R}_p^\perp, 
\]
where ${\cal R}_p = ({\text{\normalfont range }} \pi_{\ast p})$ denotes the
range of $\pi_\ast$ and its orthogonal complementary space is ${\cal R}%
_p^\perp= ({\text{\normalfont range }} \pi_{\ast p})^\perp$ in the tangent
space $T_{\pi(p)} N$. Then, the map $\pi$ is said to be a {\it Riemannian map%
}, if for all $X, Y \in \Gamma({\cal H})$ the following equation satisfies 
\cite{Fischer_1992}: 
\begin{equation}  \label{riemannian_map}
g_1(X, Y) = g_2(\pi_\ast X, \pi_\ast Y).
\end{equation}
The map $\pi_\ast$ can be viewed as a section of the bundle {\rm Hom}$%
(TM,\pi^{-1}TN)$ $\to M$, where $\pi^{-1}TN$ is the pullback bundle. The
bundle {\rm Hom}$(TM,\pi^{-1}TN)$ has a connection $\nabla$ induced by $%
\nabla^M$. With this, the symmetric {\it second fundamental form} of $\pi$
is given by \cite{Nore_1986} 
\begin{equation}  \label{eqn_sff}
(\nabla \pi_\ast) (\vartheta_1,\vartheta_2) = \nabla_{\pi_\ast
\vartheta_1}^N \pi_\ast \vartheta_2 - \pi_\ast({\nabla}_{{\vartheta_1}}^M
\vartheta_2),
\end{equation}
for all $\vartheta_1, \vartheta_2 \in \Gamma(TM)$. In addition, for $X, Y
\in \Gamma({\cal H})$, we have $(\nabla \pi_\ast) (X,Y)\in \Gamma({\cal R}%
^\perp)$ \cite{Sahin_2010}. Moreover, its trace gives the {\it tension field}
$\tau^{{\cal H}}$ of ${\cal H}$. In other words, suppose $\{h_i\}_{i=1}^{r}$
is an orthonormal basis of ${\cal H}$, then by \cite{Sahin_book} 
\[
{\rm trace}~B^{{\cal H}} = \sum\limits_{i=1}^{r} ((\nabla \pi_{\ast})(h_i,
h_i)). 
\]

For any vector field $X$ on $M$ and any section $V$ of ${\cal R}^{\perp }$,
we have $\nabla _{X}^{\pi \bot }V$, which is the orthogonal projection of $%
\nabla _{X}^{N}V$ on ${\cal R}^{\perp }$, where $\nabla ^{\pi \bot }$ is a
linear connection on ${\cal R}^{\perp }$ such that $\nabla ^{\pi \bot
}g_{2}=0$. Then for $\pi $, {\it shape operator} ${\cal S}_{V}$ is defined
as \cite[p. 188]{Sahin_book}: 
\[
\nabla _{\pi _{\ast }X}^{N}V=-{\cal S}_{V}\pi _{\ast }X+\nabla _{X}^{\pi
\bot }V,
\]%
At $p\in M$, we have $\nabla _{\pi _{\ast }X}^{N}V(p)\in T_{\pi (p)}N$, $%
{\cal S}_{V}\pi _{\ast }X\in \pi _{\ast p}(T_{p}M)$ and $\nabla _{X}^{\pi
\bot }V(p)\in (\pi _{\ast p}(T_{p}M))^{\bot }$. Furthermore, the Gauss
equation for $\pi $ is defined as \cite[p. 189]{Sahin_book} {\small 
\begin{eqnarray}
g_{2}\left( R^{{M_{2}}}\left( \pi _{\ast }Z_{1},\pi _{\ast }Z_{2}\right) \pi
_{\ast }Z_{3},\pi _{\ast }Z_{4}\right)  &=&g_{1}\left( R^{{M_{1}}}\left(
Z_{1},Z_{2}\right) Z_{3},Z_{4}\right) +g_{2}\left( \left( \nabla \pi _{\ast
}\right) \left( Z_{1},Z_{3}\right) ,\left( \nabla \pi _{\ast }\right) \left(
Z_{2},Z_{4}\right) \right)   \nonumber \\
&&-g_{2}\left( \left( \nabla \pi _{\ast }\right) \left( Z_{1},Z_{4}\right)
,\left( \nabla \pi _{\ast }\right) \left( Z_{2},Z_{3}\right) \right) ,
\label{Gauss eq-RM}
\end{eqnarray}%
}where $Z_{i}\in \Gamma \left( \ker \pi _{\ast }\right) ^{\perp }$. Here, $%
R^{{M_{1}}}$ and $R^{{M_{2}}}$ denote the curvature tensors of $M_{1}$ and $%
M_{2}$, respectively.\newline
At a point $p\in M_{1}$, suppose that $\{Z_{i}\}_{i=1}^{r}$ is an
orthonormal basis of the horizontal space. Then scalar curvatures $2~scal^{%
{\cal H}}$ and $2~scal^{{\cal R}}$ on the horizontal and range spaces are
given, respectively, by 
\begin{equation}
2{~{\rm scal}}^{{\cal H}}=\sum\limits_{i,j=1}^{r}g_{1}\left( R^{{M_{1}}%
}(Z_{i},Z_{j})Z_{j},Z_{i}\right) ,\ 2{~{\rm scal}}^{{\cal R}%
}=\sum\limits_{i,j=1}^{r}g_{2}\left( R^{{M_{2}}}(\pi _{\ast }Z_{i},\pi
_{\ast }Z_{j})\pi _{\ast }Z_{j},\pi _{\ast }Z_{i}\right) .  \label{2scalH}
\end{equation}%
\begin{equation}
2{~{\rm scal}}^{{\cal H}}\left( L_{j}\right)
=\sum\limits_{i,j=1}^{r_{j}}g_{1}\left( R^{{M_{1}}}(Z_{i},Z_{j})Z_{j},Z_{i}%
\right) ,\ 2{~{\rm scal}}^{{\cal R}}\left( \pi _{\ast }L_{j}\right)
=\sum\limits_{i,j=1}^{r_{j}}g_{2}\left( R^{{M_{2}}}(\pi _{\ast }Z_{i},\pi
_{\ast }Z_{j})\pi _{\ast }Z_{j},\pi _{\ast }Z_{i}\right) .
\label{eq-scalh-scalR-LJ-1}
\end{equation}%
Supposing $\{V_{r+1},\dots ,V_{m_{2}}\}$ an orthonormal basis of $\left( 
{\rm range~}\pi _{\ast }\right) ^{\perp }$ we set, 
\begin{eqnarray}
B_{ij}^{{\cal H}^{\alpha }} &=&g_{2}\left( (\nabla \pi _{\ast
})(Z_{i},Z_{j}),V_{\alpha }\right) ,\quad i,j=1,\dots ,r,\quad \alpha
=r+1,\dots ,m_{2},  \label{eq-Bij} \\
\left\Vert B^{{\cal H}}\right\Vert ^{2} &=&\sum_{i,j=1}^{r}g_{2}\left(
(\nabla \pi _{\ast })(Z_{i},Z_{j}),(\nabla \pi _{\ast })(Z_{i},Z_{j})\right)
,  \label{eq-norm-BH} \\
{\rm trace\,}B^{{\cal H}} &=&\sum_{i=1}^{r}(\nabla \pi _{\ast })\left(
Z_{i},Z_{i}\right) .  \label{eq-trace-BH}
\end{eqnarray}

\section{Generalized Chen's inequality for Riemannian maps \label{sec 5}}

\begin{theorem}
\label{Theorem RM 1}Let $\pi :\left( M_{1},g_{1}\right) \rightarrow \left(
M_{2},g_{2}\right) $ be a Riemannian map between Riemannian manifolds. Let $%
r_{1},\ldots ,r_{k}$ be intigers $\geq 2$ satisfying $r_{1}<r$, $%
r_{1}+\cdots +r_{k}\leq r$. For $p\in M_{1}$, let $L_{j}$ be an $r_{j}$%
-plane section of $T_{p}M_{1}$, $j=1,\ldots ,k$. Then we have%
\begin{equation}
{\rm scal}^{{\cal H}}-\sum_{j=1}^{k}{\rm scal}^{{\cal H}}\left( L_{j}\right)
\leq {\rm scal}^{{\cal R}}-\sum_{j=1}^{k}{\rm scal}^{{\cal R}}\left( \pi
_{\ast }L_{j}\right) +c\left( r_{1},\ldots ,r_{k}\right) \left\Vert {\rm %
trace\,}B^{{\cal H}}\right\Vert ^{2},  \label{eq-GCFI-RM}
\end{equation}%
for any $k$-tuple $\left( r_{1},\ldots ,r_{k}\right) \in S\left( r\right) $,
and equality holds at $p\in M_{1}$ if and only if there exist orthonormal
bases $\{h_{1},\ldots ,h_{s}\}$, $\{\pi _{\ast }h_{1},\ldots ,\pi _{\ast
}h_{s}\}$ and $\{V_{r+1},\ldots ,V_{n}\}$ of ${\cal H}_{p}$, ${\cal R}$ and $%
{\cal R}^{\perp }$, respectively, such that%
\begin{eqnarray*}
\sum_{\alpha _{1}\in D_{1}}a_{\alpha _{1}} &=&\sum_{\alpha _{2}\in
D_{2}}a_{\alpha _{2}}=\cdots =\sum_{\alpha _{r}\in D_{r}}a_{\alpha
_{r}}=a_{_{\left( \sum_{j=1}^{k}r_{j}\right) +1}}=\cdots =a_{r}, \\
\left( B^{{\cal H}}\right) _{\lambda \mu }^{r+1} &=&0,\quad \left( \lambda
,\mu \neq \lambda \right) \notin D^{2},\quad \lambda ,\mu \in \left\{
1,\ldots ,r\right\} , \\
\left( B^{{\cal H}}\right) _{\lambda \mu }^{\alpha } &=&0,\quad \alpha
=r+2,\ldots ,n,\quad \lambda ,\mu \in \left\{ 1,\ldots ,r\right\} , \\
\sum_{\alpha _{j}\in D_{j}}B_{\alpha _{j}\alpha _{j}}^{{\cal H}^{\alpha }}
&=&0,\quad \alpha =r+2,\ldots ,n,\quad j=1,\ldots k.
\end{eqnarray*}
\end{theorem}
\begin{proof}
Let $\left\{ h_{1},\ldots ,h_{r}\right\} $ and $\left\{ \pi
_{\ast }h_{1},\ldots ,\pi _{\ast }h_{r}\right\} $ be orthonormal bases of $%
{\cal H}$ and ${\cal R}$, respectively, the ${\rm trace\,}B^{{\cal H}}$ lies
along the the direction of $V_{r+1}$. For convenience we set%
\begin{equation}
a_{i}=\left( B^{{\cal H}}\right) _{ii}^{r+1}=g_{2}\left( (\nabla \pi _{\ast
})(h_{i},h_{i}),V_{r+1}\right) ,\quad i=1,\ldots ,r,  \label{eq-RM-1}
\end{equation}%
Define index sets 
\begin{eqnarray*}
D_{1} &=&\left\{ 1,\ldots ,r_{1}\right\} , \\
D_{2} &=&\left\{ r_{1}+1,\ldots ,r_{1}+r_{2}\right\}  \\
&&\vdots  \\
D_{k} &=&\left\{ r_{1}+\cdots +r_{k-1}+1,\ldots ,r_{1}+\cdots
+r_{k-1}+r_{k}\right\} 
\end{eqnarray*}%
Let $L_{1},\ldots ,L_{k}$ be mutually orthogonal subspaces of ${\cal R}$,
with $\dim L_{j}=r_{j}$, defined by%
\[
L_{j}={\rm span}\{e_{r_{1}+\cdots +r_{j-1}+1},...,e_{r_{1}+\cdots
+r_{j-1}+r_{j}}\},\quad j=1,...,k.
\]%
Then from (\ref{Gauss eq-RM}), (\ref{2scalH}), (\ref{eq-norm-BH}) and (\ref%
{eq-trace-BH}), we obtain%
\begin{equation}
2{~{\rm scal}}^{{\cal H}}=2{~{\rm scal}}^{{\cal R}}-\left\Vert B^{{\cal H}%
}\right\Vert ^{2}+\left\Vert {\rm trace\,}B^{{\cal H}}\right\Vert ^{2}
\label{eq-scalh-scalR}
\end{equation}%
we assume%
\begin{equation}
\varepsilon =2{~{\rm scal}}^{{\cal H}}-2{~{\rm scal}}^{{\cal R}}-2c\left(
r_{1},\ldots ,r_{k}\right) \left\Vert {\rm trace\,}B^{{\cal H}}\right\Vert
^{2}.  \label{eq-epsilon}
\end{equation}%
From (\ref{eq-scalh-scalR}) and (\ref{eq-epsilon}), we obtain 
\begin{equation}
\left\Vert {\rm trace\,}B^{{\cal H}}\right\Vert ^{2}=\gamma \left(
\varepsilon +\left\Vert B^{{\cal H}}\right\Vert ^{2}\right) 
\label{eq-norm-traceBH-epsilon}
\end{equation}%
where%
\[
\gamma =\left( r+k-\sum_{j=1}^{k}r_{j}\right) ,\ 2c\left( r_{1},\ldots
,r_{k}\right) =\frac{\left( r+k-1-\sum_{j=1}^{k}r_{j}\right) }{\left(
r+k-\sum_{j=1}^{k}r_{j}\right) }
\]%
Since ${\rm trace\,}B^{{\cal H}}$ lies in the direction of $V_{r+1}$, then
from (\ref{eq-norm-traceBH-epsilon}), we obtain 
\begin{equation}
\left( \sum_{i=1}^{r}B_{ii}^{{\cal H}^{r+1}}\right) ^{2}=\gamma \left(
\sum_{i=1}^{r}\left( B_{ii}^{{\cal H}^{r+1}}\right) ^{2}+\sum_{i\neq
j=1}^{r}\left( B_{ij}^{{\cal H}^{r+1}}\right) ^{2}+\sum_{\alpha
=r+2}^{n}\sum_{i,j=1}^{r}\left( B_{ij}^{{\cal H}^{\alpha }}\right)
^{2}+\varepsilon \right)   \label{eq-BH-dir-VR1}
\end{equation}%
From (\ref{eq-RM-1}) and (\ref{eq-BH-dir-VR1}), we obtain%
\begin{equation}
\left( \sum_{i=1}^{r}a_{i}\right) ^{2}=\gamma \left(
\sum_{i=1}^{r}a_{i}^{2}+\sum_{i\neq j=1}^{r}\left( B_{ij}^{{\cal H}%
^{r+1}}\right) ^{2}+\sum_{\alpha =r+2}^{n}\sum_{i,j=1}^{r}\left( B_{ij}^{%
{\cal H}^{\alpha }}\right) ^{2}+\varepsilon \right)   \label{eq-ai-square}
\end{equation}%
equation (eq-ai-square) can be written as%
\begin{eqnarray}
\left( \sum_{i=1}^{r}a_{i}\right) ^{2} &=&\gamma \left\{
\sum_{i=1}^{r}a_{i}^{2}+\sum_{i\neq j=1}^{r}\left( B_{ij}^{{\cal H}%
^{r+1}}\right) ^{2}+\sum_{\alpha =r+2}^{n}\sum_{i,j=1}^{r}\left( B_{ij}^{%
{\cal H}^{\alpha }}\right) ^{2}+\varepsilon \right.   \nonumber \\
&&\left. +\sum_{2\leq \alpha _{1}\neq \beta _{1}\leq r_{1}}a_{\alpha
_{1}}a_{\beta _{1}}-\sum_{2\leq \alpha _{1}\neq \beta _{1}\leq
r_{1}}a_{\alpha _{1}}a_{\beta _{1}}+\sum_{\alpha _{2}\neq \beta
_{2}}a_{\alpha _{2}}a_{\beta _{2}}\right.   \nonumber \\
&&\left. -\sum_{\alpha _{2}\neq \beta _{2}}a_{\alpha _{2}}a_{\beta
_{2}}+\cdots +\sum_{\alpha _{k}\neq \beta _{k}}a_{\alpha _{k}}a_{\beta
_{k}}-\sum_{\alpha _{k}\neq \beta _{k}}a_{\alpha _{k}}a_{\beta _{k}}\right\} 
\label{eq-3-di}
\end{eqnarray}%
for $\alpha _{2},\beta _{2}\in D_{2},\ldots ,\alpha _{k},\beta _{k}\in D_{k}$%
. Denoting%
\begin{eqnarray}
b_{1} &=&a_{1},\ b_{2}=a_{2}+\cdots +a_{r_{1}},\ b_{3}=a_{r_{1}+1}+\cdots
+a_{r_{1}+r_{2}},\cdots   \nonumber \\
b_{k+1} &=&a_{r_{1}+\cdots +r_{k-1}+1}+\cdots +a_{r_{1}+\cdots
+r_{k-1}+r_{k}},\quad b_{k+2}=b_{\left( k+1\right) +1}=a_{\left(
\sum_{j=1}^{k}r_{j}\right) +1}  \nonumber \\
b_{\gamma +1} &=&b_{\left( r+k-\sum_{j=1}^{k}r_{j}\right) +1}=b_{k+1+\left(
r-\sum_{j=1}^{k}r_{j}\right) }=a_{\left( \sum_{j=1}^{k}r_{j}\right) +\left(
r-\sum_{j=1}^{k}r_{j}\right) }=a_{r}  \label{eq-assuming-bi}
\end{eqnarray}%
From (\ref{eq-3-di}) and (\ref{eq-assuming-bi}), we obtain%
\begin{eqnarray}
\left( \sum_{i=1}^{\gamma +1}b_{i}\right) ^{2} &=&\gamma \left\{
\sum_{i=1}^{\gamma }b_{i}^{2}+\sum_{i\neq j=1}^{r}\left( B_{ij}^{{\cal H}%
^{r+1}}\right) ^{2}+\sum_{\alpha =r+2}^{n}\sum_{i,j=1}^{r}\left( B_{ij}^{%
{\cal H}^{\alpha }}\right) ^{2}+\varepsilon \right.   \nonumber \\
&&\left. -\sum_{2\leq \alpha _{1}\neq \beta _{1}\leq r_{1}}a_{\alpha
_{1}}a_{\beta _{1}}-\sum_{\alpha _{2}\neq \beta _{2}}a_{\alpha _{2}}a_{\beta
_{2}}-\cdots -\sum_{\alpha _{k}\neq \beta _{k}}a_{\alpha _{k}}a_{\beta
_{k}}\right\}   \label{eq-bi-square}
\end{eqnarray}%
Applying Lemma \ref{Lemma 2}\ in (\ref{eq-bi-square}), we obtain%
\begin{eqnarray}
2b_{1}b_{2} &\geq &\left\{ \sum_{i\neq j=1}^{r}\left( B_{ij}^{{\cal H}%
^{r+1}}\right) ^{2}+\sum_{\alpha =r+2}^{n}\sum_{i,j=1}^{r}\left( B_{ij}^{%
{\cal H}^{\alpha }}\right) ^{2}+\varepsilon \right.   \nonumber \\
&&\left. -\sum_{2\leq \alpha _{1}\neq \beta _{1}\leq r_{1}}a_{\alpha
_{1}}a_{\beta _{1}}-\sum_{\alpha _{2}\neq \beta _{2}}a_{\alpha _{2}}a_{\beta
_{2}}-\cdots -\sum_{\alpha _{k}\neq \beta _{k}}a_{\alpha _{k}}a_{\beta
_{k}}\right\}   \label{eq-2b1b2}
\end{eqnarray}%
equation (\ref{eq-2b1b2}) can be written as%
\begin{eqnarray}
&&\sum_{\alpha _{1}<\beta _{1}}a_{\alpha _{1}}a_{\beta _{1}}+\sum_{\alpha
_{2}<\beta _{2}}a_{\alpha _{2}}a_{\beta _{2}}+\cdots +\sum_{\alpha
_{k}<\beta _{k}}a_{\alpha _{k}}a_{\beta _{k}}  \nonumber \\
&&\ \ \ \ \ \ \ \left. \geq \frac{1}{2}\sum_{i\neq j=1}^{r}\left( B_{ij}^{%
{\cal H}^{r+1}}\right) ^{2}+\frac{1}{2}\sum_{\alpha
=r+2}^{n}\sum_{i,j=1}^{r}\left( B_{ij}^{{\cal H}^{\alpha }}\right) ^{2}+%
\frac{\varepsilon }{2}\right.   \label{eq-alpha1aplha2}
\end{eqnarray}%
where $\alpha _{j},\beta _{j}\in D_{j}$. $j=1,\ldots ,k$. From (\ref{Gauss
eq-RM}), we obtain%
\begin{equation}
\sum_{\alpha _{j}<\beta _{j}}K^{{\cal H}}\left( L_{j}\right) =\sum_{\alpha
_{j}<\beta _{j}}K^{{\cal R}}\left( \pi _{\ast }L_{j}\right) +\sum_{\alpha
=r+1}^{n}\sum_{\alpha _{j}<\beta _{j}}\left( B_{\alpha _{j}\alpha _{j}}^{%
{\cal H}^{\alpha }}B_{\beta _{j}\beta _{j}}^{{\cal H}^{\alpha }}-\left(
B_{\alpha _{j}\beta _{j}}^{{\cal H}^{\alpha }}\right) ^{2}\right) ,
\label{eq-sect-LJ}
\end{equation}%
where $\alpha _{j},\beta _{j}\in D_{j}$, $j=1,\ldots ,k$. From (\ref{eq-scalh-scalR-LJ-1}) and (\ref{eq-sect-LJ}), we obtain%
\begin{eqnarray}
\sum_{j=1}^{k}{\rm scal}^{{\cal H}}\left( L_{j}\right)  &=&\sum_{j=1}^{k}%
{\rm scal}^{{\cal R}}\left( \pi _{\ast }L_{j}\right)
+\sum_{j=1}^{k}\sum_{\alpha _{j}<\beta _{j}}\left( B_{\alpha _{j}\alpha
_{j}}^{{\cal H}^{r+1}}B_{\beta _{j}\beta _{j}}^{{\cal H}^{r+1}}-\left(
B_{\alpha _{j}\beta _{j}}^{{\cal H}^{r+1}}\right) ^{2}\right)   \nonumber \\
&&+\sum_{j=1}^{k}\sum_{\alpha =r+2}^{n}\sum_{\alpha _{j}<\beta _{j}}\left(
B_{\alpha _{j}\alpha _{j}}^{{\cal H}^{\alpha }}B_{\beta _{j}\beta _{j}}^{%
{\cal H}^{\alpha }}-\left( B_{\alpha _{j}\beta _{j}}^{{\cal H}^{\alpha
}}\right) ^{2}\right)   \label{eq-scalh-scalR-LJ}
\end{eqnarray}%
Define%
\begin{eqnarray*}
D &=&D_{1}\cup D_{2}\cup \cdots \cup D_{k} \\
D^{2} &=&D_{1}\times D_{1}\cup D_{2}\times D_{2}\cup \cdots \cup D_{k}\times
D_{k}.
\end{eqnarray*}%
Substituting the value from (\ref{eq-alpha1aplha2}) into (\ref%
{eq-scalh-scalR-LJ-1}), we obtain%
\begin{eqnarray}
\sum_{j=1}^{k}{\rm scal}^{{\cal H}}\left( L_{j}\right)  &\geq &\sum_{j=1}^{k}%
{\rm scal}^{{\cal R}}\left( \pi _{\ast }L_{j}\right) +\frac{\varepsilon }{2}+%
\frac{1}{2}\sum_{i\neq j=1}^{r}\left( B_{ij}^{{\cal H}^{r+1}}\right) ^{2}+%
\frac{1}{2}\sum_{\alpha =r+2}^{n}\sum_{i,j=1}^{r}\left( B_{ij}^{{\cal H}%
^{\alpha }}\right) ^{2}  \nonumber \\
&&-\sum_{j=1}^{k}\sum_{\alpha _{j}<\beta _{j}}\left( B_{\alpha _{j}\beta
_{j}}^{{\cal H}^{r+1}}\right) ^{2}+\sum_{j=1}^{k}\sum_{\alpha
=r+2}^{n}\sum_{\alpha _{j}<\beta _{j}}B_{\alpha _{j}\alpha _{j}}^{{\cal H}%
^{\alpha }}B_{\beta _{j}\beta _{j}}^{{\cal H}^{\alpha
}}-\sum_{j=1}^{k}\sum_{\alpha =r+2}^{n}\sum_{\alpha _{j}<\beta _{j}}\left(
B_{\alpha _{j}\beta _{j}}^{{\cal H}^{\alpha }}\right) ^{2}  \label{eq-RM-2}
\end{eqnarray}%
From (\ref{eq-RM-2}), we obtain%
\begin{eqnarray}
\sum_{j=1}^{k}{\rm scal}^{{\cal H}}\left( L_{j}\right)  &\geq &\sum_{j=1}^{k}%
{\rm scal}^{{\cal R}}\left( \pi _{\ast }L_{j}\right) +\frac{\varepsilon }{2}+%
\frac{1}{2}\sum_{\left( \lambda ,\mu \neq \lambda \right) \notin
D^{2}}\left( B_{\lambda \mu }^{{\cal H}^{r+1}}\right) ^{2}  \nonumber \\
&&+\frac{1}{2}\sum_{\alpha =r+2}^{n}\sum_{i,j=1}^{r}\left( B_{ii}^{{\cal H}%
^{\alpha }}\right) ^{2}+\sum_{j=1}^{k}\sum_{\alpha =r+2}^{n}\sum_{\alpha
_{j}<\beta _{j}}B_{\alpha _{j}\alpha _{j}}^{{\cal H}^{\alpha }}B_{\beta
_{j}\beta _{j}}^{{\cal H}^{\alpha }}-\sum_{j=1}^{k}\sum_{\alpha
=r+2}^{n}\sum_{\alpha _{j}<\beta _{j}}\left( B_{\alpha _{j}\beta _{j}}^{%
{\cal H}^{\alpha }}\right) ^{2}  \label{eq-RM-3}
\end{eqnarray}%
By using the binomial expansion in (\ref{eq-RM-3}), we obtain%
\begin{eqnarray}
\sum_{j=1}^{k}{\rm scal}^{{\cal H}}\left( L_{j}\right)  &\geq &\sum_{j=1}^{k}%
{\rm scal}^{{\cal R}}\left( \pi _{\ast }L_{j}\right) +\frac{\varepsilon }{2}+%
\frac{1}{2}\sum_{\left( \lambda ,\mu \neq \lambda \right) \notin
D^{2}}\left( B_{\lambda \mu }^{{\cal H}^{r+1}}\right) ^{2}  \nonumber \\
&&+\frac{1}{2}\sum_{j=1}^{k}\sum_{\alpha =r+2}^{n}\left( \sum_{\alpha
_{j}\in D_{j}}B_{\alpha _{j}\alpha _{j}}^{{\cal H}^{\alpha }}\right) ^{2}+%
\frac{1}{2}\sum_{\alpha =r+2}^{n}\sum_{\left( \lambda ,\mu \right) \notin
D^{2}}\left( B_{\lambda \mu }^{{\cal H}^{\alpha }}\right) ^{2}.
\label{eq-RM-4}
\end{eqnarray}%
equation (\ref{eq-RM-4}) can be written as%
\begin{eqnarray}
\sum_{j=1}^{k}{\rm scal}^{{\cal H}}\left( L_{j}\right)  &\geq &\sum_{j=1}^{k}%
{\rm scal}^{{\cal R}}\left( \pi _{\ast }L_{j}\right) +\frac{\varepsilon }{2}+%
\frac{1}{2}\sum_{\left( \lambda ,\mu \neq \lambda \right) \notin
D^{2}}\left( B_{\lambda \mu }^{{\cal H}^{r+1}}\right) ^{2}+\frac{1}{2}%
\sum_{\alpha =r+2}^{n}\sum_{\left( \lambda ,\mu \right) \notin D^{2}}\left(
B_{\lambda \mu }^{{\cal H}^{\alpha }}\right) ^{2}  \nonumber \\
&&+\frac{1}{2}\sum_{j=1}^{k}\sum_{\alpha =r+2}^{n}\left( \sum_{\alpha
_{j}\in D_{j}}B_{\alpha _{j}\alpha _{j}}^{{\cal H}^{\alpha }}\right) ^{2}.
\label{eq-RM-5}
\end{eqnarray}%
From (\ref{eq-RM-5}), we obtain%
\begin{equation}
\sum_{j=1}^{k}{\rm scal}^{{\cal H}}\left( L_{j}\right) \geq \sum_{j=1}^{k}%
{\rm scal}^{{\cal R}}\left( \pi _{\ast }L_{j}\right) +\frac{\varepsilon }{2}
\label{eq-RM-6}
\end{equation}%
From (\ref{eq-epsilon}) and (\ref{eq-RM-6}), we obtain%
\begin{equation}
\sum_{j=1}^{k}{\rm scal}^{{\cal H}}\left( L_{j}\right) \geq \sum_{j=1}^{k}%
{\rm scal}^{{\cal R}}\left( \pi _{\ast }L_{j}\right) +{\rm scal}^{{\cal H}}-%
{\rm scal}^{{\cal R}}-c\left( r_{1},\ldots ,r_{k}\right) \left\Vert {\rm %
trace\,}B^{{\cal H}}\right\Vert ^{2}  \label{eq-RM-7}
\end{equation}%
which implies%
\begin{equation}
{\rm scal}^{{\cal H}}-\sum_{j=1}^{k}{\rm scal}^{{\cal H}}\left( L_{j}\right)
\leq {\rm scal}^{{\cal R}}-\sum_{j=1}^{k}{\rm scal}^{{\cal R}}\left( \pi
_{\ast }L_{j}\right) +c\left( r_{1},\ldots ,r_{k}\right) \left\Vert {\rm %
trace\,}B^{{\cal H}}\right\Vert ^{2}  \label{eq-RM-8}
\end{equation}%
The equality case in (\ref{eq-GCFI-RM}) holds if and only if equalities are
satisfied in (\ref{eq-2b1b2}) and (\ref{eq-RM-6}). According to Lemma \ref%
{Lemma 2}, equality in (\ref{eq-2b1b2}) holds if and only if%
\[
b_{1}+b_{2}=b_{3}=\cdots =b_{k+1}=\cdots =b_{\gamma +1}=a_{r},
\]%
that is,%
\[
\sum_{\alpha _{1}\in D_{1}}a_{\alpha _{1}}=\sum_{\alpha _{2}\in
D_{2}}a_{\alpha _{2}}=\cdots =\sum_{\alpha _{r}\in D_{r}}a_{\alpha
_{r}}=a_{_{\left( \sum_{j=1}^{k}r_{j}\right) +1}}=\cdots =a_{r},
\]%
equality in (\ref{eq-RM-6}) holds if and only if%
\begin{eqnarray*}
B_{\lambda \mu }^{{\cal H}^{r+1}} &=&0,\quad \left( \lambda ,\mu \neq
\lambda \right) \notin D^{2},\quad \lambda ,\mu \in \left\{ 1,\ldots
,r\right\}  \\
B_{\lambda \mu }^{{\cal H}^{\alpha }} &=&0,\quad \alpha =r+2,\ldots ,n,\quad
\lambda ,\mu \in \left\{ 1,\ldots ,r\right\}  \\
\sum_{\alpha _{j}\in D_{j}}B_{\alpha _{j}\alpha _{j}}^{{\cal H}^{\alpha }}
&=&0,\quad \alpha =r+2,\ldots ,n,\quad j=1,\ldots k.
\end{eqnarray*}%
Which complete the proof.
\end{proof}

\begin{theorem}
Let $\pi :\left( M_{1},g_{1}\right) \rightarrow \left( M_{2}\left( c\right)
,g_{2}\right) $ be a Riemannian map from a Riemannian manifold to a real
space form. Let $r_{1},\ldots ,r_{k}$ be intigers $\geq 2$ satisfying $%
r_{1}<r$, $r_{1}+\cdots +r_{k}\leq r$. For $p\in M_{1}$, let $L_{j}$ be an $%
r_{j}$-plane section of ${\cal H}_{p}$, $j=1,\ldots ,k$. Then we have%
\begin{equation}
{\rm scal}^{{\cal H}}-\sum_{j=1}^{k}{\rm scal}^{{\cal H}}\left( L_{j}\right)
\leq \frac{1}{2}c\left\{ r\left( r-1\right) -\sum_{j=1}^{k}r_{j}\left(
r_{j}-1\right) \right\} +c(r_{1},\ldots ,r_{k})\left\Vert {\rm trace\,}B^{%
{\cal H}}\right\Vert ^{2}.  \label{eq-GRSFRM-(1)}
\end{equation}%
Equality case follows Theorem {\rm \ref{Theorem RM 1}}.
\end{theorem}

\begin{proof}
By (\ref{eq-RSF}) (\ref{2scalH}) and (\ref%
{eq-scalh-scalR-LJ-1}), we get%
\begin{equation}
\sum_{j=1}^{k}{\rm scal}^{{\cal R}}\left( L_{j}\right) =\frac{c}{2}%
\sum_{j=1}^{k}r_{j}\left( r_{j}-1\right) ,\quad {\rm scal}^{{\cal R}}=\frac{c%
}{2}r\left( r-1\right) .  \label{eq-GRSFRM-(1.1)}
\end{equation}%
In view of (\ref{eq-GRSFRM-(1.1)}) and (\ref{eq-GCFI-RM}), we get (\ref%
{eq-GRSFRM-(1)}).
\end{proof}

\begin{theorem}
Let $\pi :\left( M_{1},g_{1}\right) \rightarrow \left( M_{2}\left( c\right)
,g_{2}\right) $ be a Riemannian map from a Riemannian manifold to a complex
space form. Let $r_{1},\ldots ,r_{k}$ be intigers $\geq 2$ satisfying $%
r_{1}<r$, $r_{1}+\cdots +r_{k}\leq r$. For $p\in M_{1}$, let $L_{j}$ be an $%
r_{j}$-plane section of ${\cal H}_{p}$, $j=1,\ldots ,k$. Then we have%
\begin{eqnarray}
{\rm scal}^{{\cal H}}-\sum_{j=1}^{k}{\rm scal}^{{\cal H}}\left( L_{j}\right)
&\leq &\frac{c}{8}\left\{ r\left( r-1\right) -\sum_{j=1}^{k}r_{j}\left(
r_{j}-1\right) \right\}   \nonumber \\
&&+\frac{3c}{8}\left\{ \left\Vert P\right\Vert ^{2}-2\sum_{j=1}^{k}\theta
\left( L_{j}\right) \right\} +c(r_{1},\ldots ,r_{k})\left\Vert {\rm trace\,}%
B^{{\cal H}}\right\Vert ^{2},  \label{eq-GCFGCSFRM}
\end{eqnarray}%
where $\theta \left( L_{j}\right) =\sum_{1\leq i<j\leq r_{j}}g\left( P\pi
_{\ast }Z_{i},\pi _{\ast }Z_{j}\right) ^{2}$. Equality case follows Theorem 
{\rm \ref{Theorem RM 1}}.
\end{theorem}

\begin{proof}
By (\ref{eq-GCSF}) (\ref{2scalH}) and (\ref%
{eq-scalh-scalR-LJ-1}), we get%
\begin{equation}
\sum_{j=1}^{k}{\rm scal}^{{\cal R}}\left( \pi _{\ast }L_{j}\right) =\frac{c}{%
8}\left\{ \sum_{j=1}^{k}r_{j}\left( r_{j}-1\right) +6\sum_{j=1}^{k}\Psi
\left( L_{j}\right) \right\} ,  \label{eq-GCFGCSFRM-(1.1)}
\end{equation}%
\begin{eqnarray}
{\rm scal}^{{\cal R}} &=&\frac{c}{8}\left\{ r\left( r-1\right)
+3\sum_{i,j=1}^{r}g\left( QV_{i},V_{j}\right) ^{2}\right\}  \nonumber \\
&=&\frac{c}{8}\left\{ r\left( r-1\right) +3\left\Vert Q\right\Vert
^{2}\right\} .  \label{eq-GCFGCSFRM-(1.2)}
\end{eqnarray}%
In view of (\ref{eq-GCFGCSFRM-(1.1)}), (\ref{eq-GCFGCSFRM-(1.2)}) and (\ref%
{eq-GCFI-RM}), we get (\ref{eq-GCFGCSFRM}).
\end{proof}

\end{document}